\documentclass[12pt]{article}
\usepackage[utf8]{inputenc}
\usepackage[toc,page]{appendix}
\usepackage{geometry}
\usepackage{bbm}
\usepackage{amsmath,amsthm,amssymb,amsfonts}
\usepackage[colorlinks, linkcolor=blue]{hyperref}
\newtheorem{theorem}{Theorem}
\newtheorem{lemma}[theorem]{Lemma}
\newtheorem{remark}[theorem]{Remark}

\newtheorem{corollary}[theorem]{Corollary}
\newtheorem{definition}[theorem]{Definition}
\numberwithin{equation}{section}
\numberwithin{theorem}{section}
\usepackage{graphicx}
\providecommand{\keywords}[1]
{
  \small	
  \textbf{Keywords } #1
}
\providecommand{\thank}[1]
{
  \small	
  \textbf{Acknowledgment } #1
}

\allowdisplaybreaks
\title{\textbf{Global well-posedness of quadratic and subquadratic half wave Schrödinger equations}}
\author{Xi Chen}
\date{}

\begin{document}

\maketitle
\begin{abstract}
We consider the following $p$ order nonlinear half wave Schrödinger equations
$$\left(i \partial_{t}+\partial_{x }^2-\left|D_{y}\right|\right) u=\pm|u|^{p-1} u$$
on the plane $\mathbb{R}^2$ with $1<p\leq 2$. This equation is considered as a toy model motivated by the study of solutions to weakly dispersive equations. In particular, the global well-posedness of this equation is a difficult problem due to the anisotropic property of the equation, with one direction corresponding to the half-wave operator, which is not dispersive.  In this paper, we prove the global well-posedness of this equation in $L_x^2 H_y^s(\mathbb{R}^2) \cap H_x^1 L_y^2(\mathbb{R}^2)$($\frac{1}{2}\leq s \leq 1$), which is the first global well-posedness result of nonlinear half wave Schrödinger equations. With the global well-posedness in the energy space for the focusing equation and the study on the solitary wave in \cite{ref1}, we complete the proof of the stability of the set of ground states. Moreover, we consider the half wave Schrödinger equations on $\mathbb{R}_{x}\times\mathbb{T}_{y}$,  which can also be called the wave guide Schrödinger equations on $\mathbb{R}_{x}\times\mathbb{T}_{y}$. Using a similar approach in the analysis of the Cauchy problem of half wave Schrödinger equations on $\mathbb{R}^2$, we can also deduce the global well-posedness of $p$ ($1<p\leq2$) order wave guide Schrödinger equations in $L_x^2 H_y^s(\mathbb{R}\times\mathbb{T}) \cap H_x^1 L_y^2(\mathbb{R}\times\mathbb{T})$  with $\frac{1}{2}\leq s \leq 1$. With the global well-posedness in the energy space for the focusing wave guide Schrödinger equations and the study on the ground states in \cite{ref18}, we complete the proof of the orbital stability of the ground states with small frequencies.
\end{abstract}
\\\\
\thank{The author is currently a PhD student at Institut de Mathématiques d'Orsay of Université Paris-Saclay, and the author would like to thank his PhD advisor Patrick Gérard for his supervision of this paper. }
\tableofcontents{}
\section{Introduction}
We consider the following $p$ order nonlinear half wave Schrödinger equation on the plane,
\begin{equation}
\label{1.1}
\begin{aligned}
i \partial_{t}u + (\partial_x^2-|D_y|) u & =\pm |u|^{p-1} u, \quad(x, y) \in \mathbb{R} \times \mathbb{R},\\
u(0,x,y) & = u_0(x,y).
\end{aligned}
\end{equation}
In this paper, if not otherwise specified, we default the order of the nonlinear term $p$ to satisfy $1<p\leq 2$ in (\ref{1.1}). Also, we denote the focusing equation by $(\ref{1.1})_{-}$ and denote the defocusing equation by $(\ref{1.1})_{+}$.\\\\
We observe that this equation enjoys the following mass and energy conservations
\begin{equation}
\label{1.2}
M(u): = \frac{1}{2}\int_{\mathbb{R}^2}|u(t, x, y)|^{2} d x d y= \frac{1}{2}\int_{\mathbb{R}^2}|u_0(x,y)|^{2} d x d y,
\end{equation}
\begin{equation}
\label{1.3}
H_{\pm}(u(t)) = H_{\pm}(u_0),
\end{equation}
where 
\begin{align*}
H_{\pm}(u(t)) :=\frac{1}{2} \int_{\mathbb{R}^2}\left(\left|\partial_{x} u(t,x, y)\right|^{2}+\left||D_{y}|^{\frac{1}{2}}u(t,x, y)\right|^2\right) d x d y \pm \frac{1}{p+1} \int_{\mathbb{R}^2}|u(t,x, y)|^{p+1} d x d y.
\end{align*}
We also introduce the following anisotropic Sobolev space 
\begin{align*}
\mathcal{H}^s : = L_x^2 H_y^s(\mathbb{R}^2) \cap H_x^1 L_y^2(\mathbb{R}^2),
\end{align*}
where $\mathcal{H}^{\frac{1}{2}}$ is called the energy space. 
\begin{remark}
In fact, the solutions to (\ref{1.1}) are invariant by the scaling 
\begin{align*}
u \mapsto u_{\lambda}(t, x, y)=\lambda^{\frac{2}{p-1}} u\left(\lambda^{2} t, \lambda x, \lambda^{2} y\right),
\end{align*}
and the relevant regularity spaces for this equation are the anisotropic Sobolev spaces $L_{x}^{2} H_{y}^{s}(\mathbb{R}^2) \cap H_{x}^{2 s} L_{y}^{2}(\mathbb{R}^2)$. In this paper, with $\frac{1}{2}<s\leq 1$, we consider the spaces $\mathcal{H}^s$ rather than the spaces $L_{x}^{2} H_{y}^{s}(\mathbb{R}^2) \cap H_{x}^{2 s} L_{y}^{2}(\mathbb{R}^2)$ since we do not want to add regularity in the $x$-direction to increase the complexity in the analysis of the Cauchy problem.
\end{remark}
\begin{remark}
\label{remark 1.21}
For $u_0 \in \mathcal{H}^{\frac{1}{2}}$, if there exists a solution $u(t) \in \mathcal{H}^{\frac{1}{2}}$ to $(\ref{1.1})_{-}$, then the $\mathcal{H}^{\frac{1}{2}}$ norm of $u(t)$ is also uniformly bounded in $t$. We observe that $\mathcal{H}^{\frac{1}{2}} \subset L_{x,y}^6(\mathbb{R}^2)$, then by Hölder's inequality, we have
\begin{align*}
\|u(t)\|_{L_{x,y}^{p+1}(\mathbb{R}^2)} \leq  \|u(t)\|_{L_{x,y}^2(\mathbb{R}^2)}^{\frac{3}{p+1}-\frac{1}{2}} \|u(t)\|_{L_{x,y}^6(\mathbb{R}^2)}^{\frac{3}{2}-\frac{3}{p+1}}
\end{align*}
By the mass and the energy conservation, we have
\begin{align*}
\frac{1}{2}\|u(t)\|_{\mathcal{H}^{\frac{1}{2}}}^2 - C\|u_0\|_{L_{x,y}^2(\mathbb{R}^2)}^{\frac{5}{2}-\frac{1}{2}p} \|u(t)\|_{\mathcal{H}^{\frac{1}{2}}}^{\frac{3}{2}p-\frac{3}{2}} \leq H(u_0) \text{ with some } C > 0.
\end{align*}
Since $\frac{3}{2}p-\frac{3}{2} <2$, we infer the uniform boundedness of $\|u(t)\|_{\mathcal{H}^{\frac{1}{2}}}$.
\end{remark}
We are also interested in the half wave Schrödinger equations on the cylinder $\mathbb{R}_{x}\times\mathbb{T}_{y}$, which can also be called the wave guide Schrödinger equations on the cylinder $\mathbb{R}_{x}\times\mathbb{T}_{y}$,
\begin{equation}
\label{1.01}
\begin{aligned}
i \partial_{t}u + (\partial_x^2-|D_y|) u & =\pm |u|^{p-1} u, \quad(x, y) \in \mathbb{R} \times \mathbb{T},\\
u(0,x,y) & = u_0(x,y).
\end{aligned}
\end{equation}
Also, if not otherwise specified, we default the order of the nonlinear term $p$ to satisfy $1<p\leq 2$ in (\ref{1.01}). Similarly as before, we denote the focusing equation by $(\ref{1.01})_{-}$ and denote the defocusing equation by $(\ref{1.01})_{+}$. We observe that (\ref{1.01}) also enjoys the mass and the energy conservations in the forms of (\ref{1.2}) and (\ref{1.3}) with integrals on $\mathbb{R}_x \times \mathbb{T}_y$.\\\\ 
In order to study the Cauchy problem of (\ref{1.01}), we introduce the following anisotropic Sobolev space
\begin{align*}
\mathcal{K}^s : = L_x^2 H_y^s(\mathbb{R}\times\mathbb{T}) \cap H_x^1 L_y^2(\mathbb{R}\times\mathbb{T}),
\end{align*}
where $\mathcal{K}^{\frac{1}{2}}$ is called the energy space.
\begin{remark}
As we explained in Remark \ref{remark 1.21}, we can also infer the $\mathcal{K}^{\frac{1}{2}}$ uniform boundedness in $t$ of the $\mathcal{K}^{\frac{1}{2}}$ solutions to $(\ref{1.01})_{-}$. 
\end{remark}
The half wave Schrödinger equation was first considered by H. Xu in \cite{ref5}. She studied the large time behavior of solutions to the cubic ($ p=3 $) wave guide Schrödinger equation on the cylinder $\mathbb{R}_x\times\mathbb{T}_y$ with small smooth initial data and obtained modified scattering result. The cubic half wave Schrödinger equation on the plane $\mathbb{R}^2$ was studied by the author in \cite{ref3}, and he has obtained the existence of wave operators with small smooth initial data and some cascade results. In \cite{ref1}, Y. Bahri, S.Ibrahim and H. Kikuchi have studied the stability of the set of ground states of the focusing half wave Schrödinger equations $(\ref{1.1})_{-}$ for $1<p<5$ with the global well-posedness assumptions in the energy space $\mathcal{H}^{\frac{1}{2}}$, and they have also shown the local well-posedness of $(\ref{1.1})$ in $L_x^2 H_y^s(\mathbb{R}^2)$ for $1<p\leq 5$ with some $s>\frac{1}{2}$. In \cite{ref18}, under the assumption of the global well-posedness of the focusing wave guide Schrödinger equations $(\ref{1.01})_{-}$ with $1<p<5$ in the energy space $\mathcal{K}^{\frac{1}{2}}$, Y. Bahri, S.Ibrahim and H. Kikuchi have shown the orbital stability of the ground states with small frequencies. For the cubic half wave Schrödinger equation on the plane in lower regularity spaces, as N. Burq, P. Gérard and N. Tzvetkov did in \cite{ref10}, one can prove that the time flow map on $H_{x}^{2 s} L_{y}^{2}(\mathbb{R}^2) \cap L_{x}^{2} H_{y}^{s}(\mathbb{R}^2)$ with $\frac{1}{4} \leq s<\frac{1}{2}$ is not $C^3$ at the origin. Also, an adaptation of the arguments from \cite{ref10} implemented in I. Kato \cite{ref11} implies the ill-posedness of the cubic half wave Schrödinger equation on the plane in $H_{x}^{2 s} L_{y}^{2}(\mathbb{R}^2) \cap L_{x}^{2} H_{y}^{s}(\mathbb{R}^2)$ with $s<\frac{1}{4}$. Moreover, the probabilistic local well-posedness in quasilinear regimes for the cubic half wave Schrödinger equation on the plane has been obtained by N. Camps, L. Gassot and S. Ibrahim in \cite{ref9}. One can see an overview of the results on the cubic half wave Schrödinger equation in \cite{ref17}.
\begin{remark}
\label{remark 1.2}
The local well-posedness result in \cite{ref1} is written as follows: Assume that $1<p\leq 5$ and $s>\frac{1}{2}$, then for any $u_0 \in L_{x}^{2} H_{y}^{s}\left(\mathbb{R}^{2}\right)$ , there exist $T_{\max} > 0$ and a unique local solution $u \in C\left(\left(-T_{\max }, T_{\max }\right) ; L_{x}^{2} H_{y}^{s}\left(\mathbb{R}^{2}\right)\right)$ to (\ref{1.1}) with $\left.u\right|_{t=0}=u_0$. However, we are confused about this statement. First of all, one cannot expect high regularity on the nonlinear term $|u|^{p-1}u$ when $p\neq 3,5$, so it is not clear whether the local well-posedness still holds for large $s$ when $p\neq 3,5$. Secondly, for $2<p\leq 5$, according to the proof of the local well-posedness result in \cite{ref1}, one can only ensure the uniqueness of the solution in $C\left(\left(-T, T\right) ; L_{x}^{2} H_{y}^{s}\left(\mathbb{R}^{2}\right)\right)\cap L^{4}\left((-T, T) ; L_{x,y}^{\infty}\left(\mathbb{R}^{2}\right)\right)$, but not in $ C\left(\left(-T, T\right) ; L_{x}^{2} H_{y}^{s}\left(\mathbb{R}^{2}\right)\right)$. See the proof of Theorem 1.6 in \cite{ref1} and Remark \ref{remark a2} in this paper for more details and explanations.
\end{remark}
\subsection{Main results}
\subsubsection{Global well-posedness}
In this paper, we study the Cauchy problem of (\ref{1.1}) and deduce following the global well-posedness result of (\ref{1.1}) in $\mathcal{H}^s$ with $\frac{1}{2}\leq s \leq 1$, and this is the first global well-posedness result of nonlinear half wave Schrödinger equations. Also, this result includes the case of the global well-posedness in the energy space $\mathcal{H}^{\frac{1}{2}}$, which solves an open problem in the study of nonlinear half wave Schrödinger equations.
\begin{theorem}
\label{theorem 1.4}
Let $\frac{1}{2}\leq s \leq 1$. Given $u_0 \in \mathcal{H}^s : = L_x^2 H_y^s(\mathbb{R}^2) \cap H_x^1 L_y^2(\mathbb{R}^2)$, then there exists a unique global solution $u \in C\left(\mathbb{R} ; \mathcal{H}^s\right)$ to $(\ref{1.1})$ with $u(0) = u_0$. Moreover, for every $T>0$, the flow map $u_0 \in \mathcal{H}^s \mapsto u \in C\left([-T, T], \mathcal{H}^s\right)$ is continuous.
\end{theorem}
Also, as a corollary of Theorem \ref{theorem 1.4}, we infer the following global well-posedness result of the quadratic ($p=2$) half wave Schrödinger equations (\ref{1.1}) in the higher order Sobolev space $H^2(\mathbb{R}^2)$.
\begin{corollary}
\label{corollary 1.5}
Let $p=2$. Given $u_0 \in H^2(\mathbb{R}^2)$, there exists a unique global solution $u \in C(\mathbb{R};H^2(\mathbb{R}^2))$ to (\ref{1.1}) with $u(0) = u_0$. Also, for every $T>0$, the flow map $u_0 \in H^2(\mathbb{R}^2) \mapsto u \in C\left([-T, T], H^2(\mathbb{R}^2)\right)$ is continuous. 
\end{corollary}
Using a similar approach in the analysis of the Cauchy problem of  (\ref{1.1}) in $\mathcal{H}^s$ with $\frac{1}{2}\leq s\leq 1$, we can also deduce the global well-posedness of (\ref{1.01}) in $\mathcal{K}^s$ with $\frac{1}{2}\leq s\leq 1$.
\begin{theorem}
\label{theorem 1.6}
Let $\frac{1}{2}\leq s \leq 1$. Given $u_0 \in \mathcal{K}^s : = L_x^2 H_y^s(\mathbb{R}\times\mathbb{T}) \cap H_x^1 L_y^2(\mathbb{R}\times\mathbb{T})$, then there exists a unique global solution $u \in C\left(\mathbb{R} ; \mathcal{K}^s\right)$ to $(\ref{1.01})$ with $u(0) = u_0$. Moreover, for every $T>0$, the flow map $u_0 \in \mathcal{K}^s \mapsto u \in C\left([-T, T], \mathcal{K}^s\right)$ is continuous.
\end{theorem}
\subsubsection{The stability result}
In \cite{ref1}, Y. Bahri, S.Ibrahim and H. Kikuchi have studied the solitary wave of $(\ref{1.1})_{-}$ with $1<p<5$.  With the assumption of the global well-posedness of $(\ref{1.1})_{-}$ in the energy space, they have shown the stability of the set of ground states. In \cite{ref18}, Y. Bahri, S.Ibrahim and H. Kikuchi have studied the ground states of $(\ref{1.01})_{-}$ with $1<p<5$, and they have shown the orbital stability of the ground states with small frequencies under the assumption of the global well-posedness of $(\ref{1.01})_{-}$ in the energy space. Since we have already confirmed the global well-posedness of $(\ref{1.1})_{-}$ and $(\ref{1.01})_{-}$ with $1<p\leq 2$ in the energy space, we can reformulate the stability results in \cite{ref1} and \cite{ref18} with $1<p\leq 2$.\\\\
\textbf{Stability result of $(\ref{1.1})_{-}$}:\\\\
We first introduce the ground states corresponding to $(\ref{1.1})_{-}$. Let $\omega>0$, we observe that $\phi_{\omega}(x,y) e^{i \omega t}$ is a solution to $(\ref{1.1})_{-}$ if and only if $\phi_{\omega}$ solves the following elliptic equation
\begin{equation}
\label{4.2}
-\partial_{x}^2 \phi_{\omega}+\left|D_{y}\right| \phi_{\omega}+\omega \phi_{\omega}-\left|\phi_{\omega}\right|^{p-1} \phi_{\omega}=0, \quad (x,y) \in \mathbb{R}^{2},
\end{equation}
where $1<p\leq 2$. \\\\
We also define the following energy functional $\mathcal{E}_{\omega}$ which is corresponding to (\ref{4.2}):
\begin{align*}
\mathcal{E}_{\omega}(v):=\frac{1}{2}\left\|\partial_{x} v\right\|_{L^{2}(\mathbb{R}^2)}^{2}+\frac{1}{2}\left\|\left|D_{y}\right|^{\frac{1}{2}} v\right\|_{L^{2}(\mathbb{R}^2)}^{2}+\frac{\omega}{2}\|v\|_{L^{2}(\mathbb{R}^2)}^{2}-\frac{1}{p+1}\|v\|_{L^{p+1}(\mathbb{R}^2)}^{p+1}.
\end{align*}
We can easily observe that $\phi_{\omega} \in \mathcal{H}^{\frac{1}{2}}$ is a solution to (\ref{4.2}) if and only if $\phi_{\omega}$ is a critical point of the fractional $\mathcal{E}_{\omega}$. The ground state $Q_{\omega} \in \mathcal{H}^{\frac{1}{2}}$ is defined as a solution to (\ref{4.2}) which is a minimizer of the corresponding functional $\mathcal{E}_{\omega}$ for all non-trivial solutions to (\ref{4.2}). According to the result in \cite{ref1}, we have already the existence of the ground states of (\ref{4.2})for every $\omega>0$.\\\\
Then we focus on the orbital stability of the ground states. The stability is defined as follows.
\begin{definition}
(i) Let $\Gamma \subset \mathcal{H}^{\frac{1}{2}}$. We say that $\Gamma$ is stable if the following statment holds: For any $\varepsilon > 0$, there exists $\delta > 0$ such that for any $u_0 \in \mathcal{H}^{\frac{1}{2}}$ which sitisfies $\inf _{v \in \Gamma}\left\|u_{0}-v\right\|_{\mathcal{H}^{\frac{1}{2}}}<\delta$, the corresponding solution $u(t)$ satifies 
\begin{align*}
\sup _{t \in \mathbb{R}} \inf _{v \in \Gamma}\|u(t)-v\|_{\mathcal{H}^{\frac{1}{2}}}<\varepsilon.
\end{align*}
(ii) Let $Q_{\omega}$ be the ground state of (\ref{4.2}). We say that $e^{i \omega t} Q_{\omega}$ is orbitally stable if its orbit
\begin{align*}
\mathcal{O}_{\omega}=\left\{e^{i \vartheta} Q_{\omega}\left(x+x_{1}, y+y_{1}\right) \mid \vartheta \in \mathbb{R},\left(x_{1}, y_{1}\right) \in \mathbb{R}^{2}\right\}
\end{align*}
is stable.
\end{definition}
We recall the definitions of $M(u)$ and $H_{-}(u)$ in (\ref{1.2}) and (\ref{1.3}). To reformulate the stability result in \cite{ref1}, we introduce the following minimization problem:
\begin{align*}
\mathcal{I}(\eta):=\inf_{M(u)=\eta\ } H_{-}(u) \text{ for each } \eta > 0.
\end{align*}
We denote by $\Gamma(\eta)$ the set of minimizers, which is
\begin{align*}
\Gamma(\eta):=\{u \in \mathcal{H}^{\frac{1}{2}} \mid H_{-}(u)=\mathcal{I}(\eta), M(u)=\eta\}.
\end{align*}
Since we have proved the global well-posedness of $(\ref{1.1})_{-}$ in the energy space $\mathcal{H}^{\frac{1}{2}}$, combined with the stability result in \cite{ref1}, we complete the proof of the following stability result of $(\ref{1.1})_{-}$.
\begin{corollary}
\label{corollary 1.7}
For $\eta > 0$, $\Gamma(\eta)$ is non-empty and $\Gamma(\eta)$ is stable under the flow of $(\ref{1.1})_{-}$.
\end{corollary}
\begin{remark}
Let $Q_{\omega}$ be the ground state to (\ref{4.2}). Then we have
\begin{align*}
Q_{\omega}(x, y)=\omega^{\frac{1}{p-1}} Q_{1}(\sqrt{\omega} x, \omega y).
\end{align*}
By setting
\begin{align*}
\omega(\eta)=\left(\frac{2 \eta}{\left\|Q_{1}\right\|_{L^{2}(\mathbb{R}^2)}^{2}}\right)^{\frac{2(p-1)}{7-3p}}
\end{align*}
for each $\eta > 0$, we observe that $Q_{\omega(\eta)} \in \Gamma(\eta)$ and $\mathcal{O}_{\omega(\eta)} \subset \Gamma(\eta)$.  Thus from Corollary \ref{corollary 1.7}, we obtain the stability of the set of ground states. Moreover, to obtain the orbital stability of $e^{i \omega t} Q_{\omega}$, we only have to prove the uniqueness of the ground state. In fact, uniqueness of the ground states of the half wave equations has been proved in \cite{ref15} and \cite{ref16}. However, due to the anisotropy of the equation $(\ref{1.1})_{-}$, it is unknown whether the method of \cite{ref15} and \cite{ref16} can be applied to the equation $(\ref{1.1})_{-}$.
\end{remark}
\noindent\textbf{Stability result of $(\ref{1.01})_{-}$}:\\\\
The ground states $Q_{\omega} \in \mathcal{K}^{\frac{1}{2}}$ corresponding to $(\ref{1.01})_{-}$ can be defined as in the previous part by replacing $\mathbb{R}^2$ with $\mathbb{R}_{x}\times\mathbb{T}_{y}$. Accroding to the result in \cite{ref18}, we have already the existence of the ground states $Q_{\omega}$ for every $\omega>0$.\\\\
Then we focus on the oribital stability of the ground states. The definition of the stability can be referred to the previous part. Since we have proved the global well-posedness of $(\ref{1.01})_{-}$ in the energy space $\mathcal{K}^{\frac{1}{2}}$, combined with the stability result in \cite{ref18}, we complete the proof of the following orbital stability result of the ground states with small frequencies.
\begin{corollary}
Let $1<p\leq 2$ and $\omega_{p} = \frac{4}{(p-1)(p+3)}$. Then there exits $\omega_{*}\in\left(0, \omega_{p}\right]$ such that the ground state standing wave $e^{i \omega t} Q_{\omega}$ is orbitally stable under the flow of $(\ref{1.01})_{-}$ when $0<\omega<\omega_{*}$.
\end{corollary}
\begin{remark}
For $1<p<5$, Y. Bahri, S.Ibrahim and H. Kikuchi have proved in \cite{ref18} that, there exists $\omega_{*}\in\left(0, \omega_{p}\right]$ such that the ground state $Q_\omega$ coincides with the line soliton
\begin{align*}
R_{\omega}(x) := \left(\frac{(p+1) \omega}{2}\right)^{\frac{1}{p-1}} \operatorname{sech}^{\frac{2}{p-1}}\left(\frac{(p-1) \sqrt{\omega}}{2} x\right), \quad x\in \mathbb{R}
\end{align*}
for all $0<\omega\leq\omega_{*}$ and that the ground state $Q_{\omega}$ depends on $y$ for all $\omega>\omega_{*}$.
\end{remark}
\subsection{Structure of the paper}
In Section 2, we first introduce the notation in this paper, and then we give the Strichartz estimate for nonlinear half wave Schrödinger equations, which is the fundamental estimate in the proof of our main theorem. We also introduce the Brezis–Gallouët type inequality to obtain a logarithm type control. These arguments will be used in the proof of the main theorem in Section 3.\\\\
In Section 3, we prove our main theorem which implies the global well-posedness of (\ref{1.1}) in $\mathcal{H}^s$ with $\frac{1}{2}\leq s \leq 1$. More precisely, with $\frac{1}{2}< s \leq 1$, we use the Strichartz estimate and the Brezis-Gallouët type inequality to prove the global well-posedness of (\ref{1.1}) in $\mathcal{H}^s$. To show the global well-posedness of (\ref{1.1}) in the energy space $\mathcal{H}^{\frac{1}{2}}$, we use the classical approach to construct the weak solution in the energy space $\mathcal{H}^{\frac{1}{2}}$, and then we use an argument in \cite{ref13} to prove the uniqueness of the weak solution. The continuity of the weak solution comes from the mass and the energy conservations. Using a similar approach in the proof of the main theorem, we also infer the global well-posedness of (\ref{1.01}) in $\mathcal{K}^s$ with $\frac{1}{2}\leq s \leq 1$. \\\\
In Section 4, we discuss the well-posedness of half wave Schrödinger equations (\ref{1.1}) with $2<p\leq 5$. In the case of $2<p\leq 5$, we still can obtain the local well-posedness in $\mathcal{H}^s$ for some $s>\frac{1}{2}$, but the global well-posedness in $\mathcal{H}^s$ is unknown for any $s>\frac{1}{2}$ since we do not have a good approach to control the $L_{x,y}^{\infty}$ norm of solutions yet. The global well-posedness in the energy space $\mathcal{H}^{\frac{1}{2}}$ is also unknown, and we cannot adapt directly the argument in Theorem \ref{theorem 3.1} to deduce the uniqueness of weak solution in this case.
\section{Preliminaries}  
\subsection{Notation}
We define the Fourier trasform on $\mathbb{R}_{x}$ by
\begin{align*}
\hat{g}(\xi):=\mathcal{F}_{x}(g)(\xi)=\frac{1}{2\pi}\int_{\mathbb{R}} \mathrm{e}^{-i x \xi} g(x) d x.
\end{align*}
Similarly, we also define the Fourier transform on $\mathbb{R}_y$ by
\begin{align*}
h_{\eta}:=\mathcal{F}_{y}(h)(\eta)=\frac{1}{2\pi}\int_{\mathbb{R}} e^{-i\eta y} h(y) d y.
\end{align*}
We also define the Fourier transform on $\mathbb{T}_y$ by
\begin{align*}
h_{N}:=\mathcal{F}_{y}(h)(N)=\frac{1}{2\pi}\int_{\mathbb{T}} e^{-iN y} h(y) d y, \quad N \in \mathbb{Z}.
\end{align*}
Then we define the full Fourier transform on $\mathbb{R}_x \times \mathbb{R}_y$ by
\begin{align*}
(\mathcal{F} U)(\xi, \eta)=\frac{1}{2\pi}\int_{\mathbb{R}} \hat{U}(\xi, y) \mathrm{e}^{-i \eta y} d y=\hat{U}_{\eta}(\xi).
\end{align*}
We recall the definition of the ``Japanese bracket",
\begin{align*}
\langle\xi\rangle :=\left(1+|\xi|^{2}\right)^{1 / 2}, \quad \forall \xi \in \mathbb{R}.
\end{align*}
\subsection{Strichartz estimate}
By a simple modification of the proof of the inhomogeneous extended Strichartz estimates in \cite{ref12}, we can deduce the following Strichartz estimate for half wave Schrodinger equations on $\mathbb{R}^2$.
\begin{lemma}
\label{lemma 2.1}
Let $\left(q_{1}, r_{1}\right),\left(q_{2}, r_{2}\right)$ be two pairs of exponents satisfying
\begin{align*}
\frac{2}{q_{j}}+\frac{1}{r_{j}}=\frac{1}{2}, \quad  q_{j}>2, \quad j= 1, 2.
\end{align*}
Let $s \geq 0$ and $u_0 \in L_x^2 H_{y}^{s}$, $I$ be an interval of $\,\mathbb{R}$ containing 0, and let $f \in L_{t}^{q_{1}^{\prime}} L_{x}^{r_{1}^{\prime}} H_{y}^{s}\left(I \times \mathbb{R}^{2}\right)$. Then the unique solution $u \in C\left(I, \mathcal{{S}}^{\prime}\left(\mathbb{R}^{2}\right)\right) $ of 
\begin{equation}
i \partial_{t} u+(\partial_x^2-|D_y|) u=f, u(0)=u_{0}
\end{equation}
belongs to $C\left(I, L_x^{2} H_{y}^{s}\left(\mathbb{R}^{2}\right)\right) \cap L_{t}^{q_{2}} L_{x}^{r_{2}} H_{y}^{s}\left(I \times \mathbb{R}^{2}\right)$, with the estimate 
\begin{equation}
\label{2.2}
\sup _{t \in I}\|u(t)\|_{L_x^{2} H_y^s(\mathbb{R}^2)}+\|u\|_{L_{t}^{q_{2}}\left(I; L_{x}^{r_{2}} H_y^s(\mathbb{R}^2)\right)} \leq C\left(\left\|u_{0}\right\|_{L_x^{2}H_y^s(\mathbb{R}^2)}+\|f\|_{L_{t}^{q_{1}^{\prime}}\left(I; L_{x}^{r_{1}^{\prime}}H_y^s(\mathbb{R}^2)\right)}\right).
\end{equation}
\end{lemma}
\begin{remark}
Lemma \ref{lemma 2.1} is the fundamental estimate in the proof of Theorem \ref{theorem 3.1}, and we use the estimates (\ref{2.2}) in the case of $\frac{1}{2} <s \leq 1$ and $s=0$.
\end{remark}
\begin{remark}
\label{remark 2.3}
In fact, we have also the Strichartz estimate for wave guide Schrödinger equations on $\mathbb{R}_{x}\times\mathbb{T}_{y}$, which has the same form as Lemma \ref{lemma 2.1} by replacing $\mathbb{R}^2$ wih $\mathbb{R}_{x}\times\mathbb{T}_{y}$.
\end{remark}
\subsection{Brezis–Gallouët type inequality}
A basic tool along the proof of Theorem \ref{theorem 3.1} will be the following inequality, which is similar to the inequality obtained by H. Brezis and T. Gallouët in \cite{ref2}. 
\begin{lemma}
Given $s>\frac{1}{2}$, there exists $C_s > 0$ such that, $\forall v \in H^s(\mathbb{R})$, 
\begin{equation}
\label{2.3}
\|v_{\eta}\|_{L_{\eta}^1 L_x^2 \left(\mathbb{R}^2\right)} \leq C_{s}\|v\|_{L_{x}^{2} H_{y}^{\frac{1}{2}}\left(\mathbb{R}^2\right)}\left[\log \left(1+\frac{\|v\|_{L_{x}^{2} H_y^{s}(\mathbb{R}^2)}}{\|v\|_{L_{x}^{2} H_y^{\frac{1}{2}}(\mathbb{R}^2)}}\right)\right]^{\frac{1}{2}}.
\end{equation}
\end{lemma}
\begin{proof}
We have, for every $R>0$,
\begin{align*}
\|v_\eta\|_{L_{\eta}^1 L_x^2\left(\mathbb{R}^2\right)} \leq \int_{\mathbb{R}} \|v_\eta\|_{L_x^2} d \eta = N_1(R) + N_2(R),
\end{align*}
\begin{align*}
N_{1}(R)=\int_{|\eta| \leq R}\|v_\eta\|_{L_x^2} d \eta, \quad  N_{2}(R)=\int_{|\eta|>R}\|v_\eta\|_{L_x^2}d \eta.
\end{align*}
By the Cauchy-Schwarz inequality,
\begin{align*}
N_{2}(R) \leq\|v\|_{L_x^2 H_y^{s} \left(\mathbb{R}^2\right)} \left(\int_{|\eta|>R}\langle\eta\rangle^{-2 s} d \eta\right)^{\frac{1}{2}} \leq K_s R^{-(s-\frac{1}{2})}
\|v\|_{L_x^2 H_y^{s} \left(\mathbb{R}^2\right)},
\end{align*}
\begin{align*}
N_{1}(R) \leq\|v\|_{L_x^2 H_y^{\frac{1}{2}} \left(\mathbb{R}^2\right)}\left( \int_{|\eta| \leq R}\langle\eta\rangle^{-1} d \eta\right)^{\frac{1}{2}} \leq A\left[\log \left(1+R\right)\right]^{1 / 2}\|v\|_{L_x^2 H_y^{\frac{1}{2}} \left(\mathbb{R}^2\right)}.
\end{align*}
The inequality follows by choosing $R$ such that
\begin{align*}
R^{(s-\frac{1}{2})}=\frac{\|v\|_{L_x^2 H_y^{s} \left(\mathbb{R}^2\right)}}{\|v\|_{L_x^2 H_y^{\frac{1}{2}} \left(\mathbb{R}^2\right)}}.
\end{align*}
\end{proof}
\begin{remark}
\label{remark 2.6}
We have also the Brezis–Gallouët type inequality on $\mathbb{R}_{x}\times\mathbb{T}_{y}$
\begin{equation}
\label{2.31}
\|v_{N}\|_{\ell_{N}^1 L_x^2 \left(\mathbb{R}\right)} \leq C_{s}\|v\|_{L_{x}^{2} H_{y}^{\frac{1}{2}}\left(\mathbb{R}\times\mathbb{T}\right)}\left[\log \left(1+\frac{\|v\|_{L_{x}^{2} H_y^{s}(\mathbb{R}\times\mathbb{T})}}{\|v\|_{L_{x}^{2} H_y^{\frac{1}{2}}(\mathbb{R}\times\mathbb{T})}}\right)\right]^{\frac{1}{2}}.
\end{equation}
The proof of (\ref{2.31}) is similar to the proof of (\ref{2.3}).
\end{remark}
\begin{remark}
In fact, we have the following 2D Brezis-Gallouët inequality: \\
Given $s>1$, there exists $C_s > 0$ such that
\begin{equation}
\label{2.41}
\|\hat{v}_{\eta}(\xi)\|_{L_{\xi,\eta}^1\left(\mathbb{R}^{2}\right)} \leq C_{s}\|v\|_{H^{1}\left(\mathbb{R}^{2}\right)}\left[\log \left(1+\frac{\|v\|_{H^{s}\left(\mathbb{R}^{2}\right)}}{\|v\|_{H^{1}\left(\mathbb{R}^{2}\right)}}\right)\right]^{\frac{1}{2}}, \quad \forall v \in H^{s}\left(\mathbb{R}^{2}\right).
\end{equation}
The proof of (\ref{2.41}) is similar to the proof of (\ref{2.3}), and (\ref{2.41}) will be used in the proof of Corollary \ref{corollary 3.2}.
\end{remark}
\subsection{Trudinger type estimate}
In \cite{ref4}, P. Gérard and S. Grellier adapted the approach in \cite{ref14} to deduce the Trudinger type estimate (\ref{2.7}) for $H^{\frac{1}{2}}(\mathbb{T})$, and we use the same method to deduce the following Trudinger type estimate for $H^{\frac{1}{2}}(\mathbb{R})$. The Trudinger type estimate for $H^{\frac{1}{2}}(\mathbb{R})$ will be used in the proof of the uniqueness of weak solution to (\ref{1.1}) in Theorem \ref{theorem 3.1}.
\begin{lemma}
\label{lemma 2.4}
For $u \in H^{\frac{1}{2}}(\mathbb{R})$, we have
\begin{equation}
\label{2.4}
\|u\|_{L^{k}(\mathbb{R})} \leq C \sqrt{k} \|u\|_{H^{\frac{1}{2}}(\mathbb{R})}, \quad \forall k \in (2,\infty).
\end{equation}
\end{lemma}
\begin{proof}
Without any loss of generality, we can assume that
\begin{align*}
\|u\|_{H^{\frac{1}{2}}(\mathbb{R})} = 1.
\end{align*}
Let 
\begin{align*}
u = \mathbf{1}_{|D|<t} u+\mathbf{1}_{|D| \geq t} u
\end{align*}
for a well chosen positive number $t$. Indeed,
\begin{align*}\left\|\mathbf{1}_{|D|<t} u\right\|_{L^{\infty}} & \leq \int_{|\xi|<t}|\hat{u}(\xi)| d \xi \\ &  \leq  \left(\int_{|\xi|<t} \langle\xi\rangle^{-1} d\xi\right)^{\frac{1}{2}}\left(\int_{|\xi|<t} \langle\xi\rangle|\hat{u}(\xi)|^{2} d \xi\right)^{\frac{1}{2}} \\ & \leq C_1 [\log(1+t)]^{\frac{1}{2}}. 
\end{align*}
We deine
\begin{align*}
t(\lambda) = \mathrm{e}^{\frac{\lambda^2}{4C_1^2}}-1,
\end{align*}
so that
\begin{align*}
\left\|\mathbf{1}_{|D|<t(\lambda)} u\right\|_{L^{\infty}} \leq \frac{\lambda}{2}.
\end{align*}
Since 
\begin{align*}
\{|u|>\lambda\} \subset\left\{\left|\mathbf{1}_{|D|<t} u\right|>\frac{\lambda}{2}\right\} \cup\left\{\left|\mathbf{1}_{|D| \geq t} u\right|>\frac{\lambda}{2}\right\},
\end{align*}
we conclude that
\begin{align*}
|\{|u|>\lambda\}| \leq\left|\left\{\left|\mathbf{1}_{|D| \geq t(\lambda)} u\right|>\frac{\lambda}{2}\right\}\right|.
\end{align*}
By the Markov inequality,
\begin{align*}
\left|\left\{\left|\mathbf{1}_{|D| \geq t(\lambda)} u\right|>\frac{\lambda}{2}\right\}\right| & \leq \frac{4}{\lambda^{2}} \int_{\mathbb{R}}\left|\mathbf{1}_{|D| \geq t(\lambda)} u\right|^{2} d x \\ & \leq \frac{4}{ \lambda^{2}} \int_{|\xi|>t(\lambda)}|\hat{u}(\xi)|^{2} d \xi
\end{align*}
where we have used the Plancherel formula. Then we have
\begin{align*}
\|u\|_{L^{k}}^{k} & = k \int_{0}^{\infty} \lambda^{k-1}|\{|u|>\lambda\}| d \lambda \\& \leq 4k \int_{0}^{\infty} \lambda^{k-3} \left(\int_{|\xi|>t(\lambda)}|\hat{u}(\xi)|^{2} d \xi\right) d \lambda \\ & \leq 4k \int_{\mathbb{R}}|\hat{u}(\xi)|^{2}\left(\int_{t(\lambda)<|\xi|} \lambda^{k-3} d \lambda\right) d \xi \\ & \leq \frac{4k}{k-2} (2C_1)^{k-2} \int_{\mathbb{R}} [\log(2+|\xi|)]^{\frac{k-2}{2}} |\hat{u}(\xi)|^{2} d\xi.
\end{align*}
In fact, with $\ell \in \mathbb{Z}_{>0}$, we have
\begin{align*}
[\log(2+|\xi|)]^{\ell} \lesssim \ell!(2+|\xi|) \lesssim \ell^\ell \langle\xi\rangle.
\end{align*}
It gives the expected constant proportional to $k$ in (\ref{2.4}).
\end{proof}
\begin{remark}
\label{remark 2.8}
The Trudinger type estimate for $H^{\frac{1}{2}}(\mathbb{T})$ reads as follows
\begin{equation}
\label{2.7}
\|u\|_{L^k(\mathbb{T})} \leq C \sqrt{k} \|u\|_{H^{\frac{1}{2}}(\mathbb{T})},
\end{equation}
and this estimate will be adapted in the proof of the uniqueness of weak solution to (\ref{1.01}).
\end{remark}
\section{Global well-posedness}
In this section, we study the Cauchy problem of (\ref{1.1}) and show the global well-posedness result of (\ref{1.1}). By a similar approach, we also infer the global well-posedness result of (\ref{1.01}). \\\\
Firstly, we prove the following theorem which implies the global well-posedness of (\ref{1.1}) in $\mathcal{H}^s$ with $\frac{1}{2}\leq s \leq 1$.
\begin{theorem}
\label{theorem 3.1}
Let $\frac{1}{2}\leq s \leq 1$. Given $u_0 \in \mathcal{H}^s : = L_x^2 H_y^s(\mathbb{R}^2) \cap H_x^1 L_y^2(\mathbb{R}^2)$, then there exists a unique global solution $u \in C\left(\mathbb{R} ; \mathcal{H}^s\right)$ to $(\ref{1.1})$ with $u(0) = u_0$. Moreover, for every $T>0$, the flow map $u_0 \in \mathcal{H}^s \mapsto u \in C\left([-T, T], \mathcal{H}^s\right)$ is continuous.
\end{theorem}
\begin{proof}
We decompose the proof into two parts:\\
\textbf{Case 1} is the proof of the global well-posedness of (\ref{1.1}) in $\mathcal{H}^s$ with $\frac{1}{2}< s \leq 1$.\\
\textbf{Case 2} is the proof of the global well-posedness of (\ref{1.1}) in $\mathcal{H}^{\frac{1}{2}}$.\\\\
\textbf{Case 1: } $\frac{1}{2}< s \leq 1$. \\
As a first step, we shall prove the existence of the local solution to (\ref{1.1}) with $u_0 \in \mathcal{H}^s$ ($\frac{1}{2}< s \leq 1$). We set
\begin{align*}
 X_{T} & :=C\left([-T, T]; \mathcal{H}^s\right) , \\ \|u\|_{X_{T}} & :=\max _{t \in[-T, T]}\|u(t)\|_{\mathcal{H}^s}.
\end{align*}
We also define
\begin{equation}
\Phi_{p}(u)(t):=\mathrm{e}^{i t (\partial_x^2-|D_y|)} u_{0} \pm i \int_{0}^{t} \mathrm{e}^{i\left(t-t^{\prime}\right) (\partial_x^2-|D_y|)}\left(\left|u\left(t^{\prime}\right)\right|^{p-1}u\left(t^{\prime}\right)\right) d t^{\prime}.
\end{equation}
From Lemma \ref{lemma 2.1} and Lemma \ref{lemma a1}, we have 
\begin{align*}
\|\Phi_{p}(u)\|_{X_{T}} & \leq C\left\|u_{0}\right\|_{\mathcal{H}^s}+C\left\|\left|u\right|^{p-1} u\right\|_{{L_t^{q'}\left([-T, T];   L_x^{r'} H_y^s\left(\mathbb{R}^{2}\right)\right)}}+ C\left\||u|^{p-1} \partial_x u\right\|_{{L_t^{q'}\left([-T, T];   L_x^{r'} L_y^2\left(\mathbb{R}^{2}\right)\right)}}   \\ & \leq C\left\|u_{0}\right\|_{\mathcal{H}^s}+C T^{\frac{1}{q'}}\|u\|_{X_{T}}^{p},
\end{align*}
where 
\begin{align*}
\frac{1}{q'} = 1-\frac{1}{q} = \frac{5-p}{4}, \quad \frac{1}{r'}= 1-\frac{1}{r} = \frac{p}{2}.
\end{align*}
Similarly, we have
\begin{align*}
\|\Phi_{p}(u)-\Phi_{p}(\tilde{u})\|_{X_{T}} \leq K T^{\frac{1}{q'}} \max \left(\|u\|_{X_{T}}^{p-1},\|\tilde{u}\|_{X_{T}}^{p-1}\right)\|u-\tilde{u}\|_{X_{T}}.
\end{align*}
Consequently, if $R>0$ is such that
\begin{align*}
R>2 C\left\|u_{0}\right\|_{\mathcal{H}^s}, \max (C, K) T^{\frac{1}{q'}} R^{p-1} \leq \frac{1}{2},
\end{align*}
then $\Phi_{p}$ is a contraction map on the closed ball $B_R$ of radius $R$ and centered at 0 in $X_T$, and therefore has a unique fixed point in $B_R$. Thus there exist $T>0$ and a local solution $u \in C\left([-T, T] ; \mathcal{H}^s\right)$ to (\ref{1.1}) with $u(0) = u_0$.\\

Next, we study the global existence of this local solution. In fact, by the energy conservation, $\|u(t)\|_{H_x^1L_y^2 (\mathbb{R}^2)}$ is uniformly bounded in $t$. Then from the above derivation, we have the blow up criterion at $T^{*}>0$
\begin{equation}
\label{3.31}
\|u(\tau)\|_{L_x^2 H_y^s (\mathbb{R}^2)} \xrightarrow{ \tau \to T^{*}} \infty.
\end{equation}
Again by Lemma \ref{lemma 2.1}, we have
\begin{align*}
& \|u(\tau)\|_{L_x^2 H_y^s\left(\mathbb{R}^{2}\right)} \\ \leq &\, C \|u_0\|_{L_x^2 H_y^s(\mathbb{R}^{2})} + C \left\|\left|u\right|^{p-1} u\right\|_{{L_t^{q'}\left((0, \tau);   L_x^{r'} H_y^s\left(\mathbb{R}^{2}\right)\right)}} \\ \leq & \, C + C \left\|\|u\|_{L_{y}^{\infty}(\mathbb{R})}^{p-1} \|u\|_{H_y^s(\mathbb{R})}\right\|_{{L_t^{q'}\left((0, \tau);   L_x^{r'} \left(\mathbb{R}\right)\right)}} \\ \leq & \,  C  + C \left\|\|u_{\eta}\|_{L_{\eta}^1(\mathbb{R})}^{p-1} \|u\|_{H_y^s(\mathbb{R})}\right\|_{{L_t^{q'}\left((0, \tau);   L_x^{r'} \left(\mathbb{R}\right)\right)}} \\ \leq & \, C +  C\left\| \|u_{\eta}\|_{ L_{x}^{2} L_{\eta}^1(\mathbb{R}^2)}^{p-1} \|u\|_{L_x^2 H_y^s(\mathbb{R}^2)}\right\|_{L_{t}^{q'}(0,\tau)}\\ \leq & \, C  + C\left\| \|u_{\eta}\|_{L_{\eta}^1 L_{x}^{2}(\mathbb{R}^2)}^{p-1} \|u\|_{L_x^2 H_y^s(\mathbb{R}^2)}\right\|_{L_{t}^{q'}(0,\tau)}.
\end{align*}
The second inequality above comes from Lemma \ref{lemma a1}, and the last inequality above comes from the Minkowski inequality. \\\\
We recall the Brezis–Gallouët type inequality (\ref{2.3}),
\begin{align*}
\|u_{\eta}\|_{L_{\eta}^1 L_x^2 \left(\mathbb{R}^2\right)} \leq C\|u\|_{L_{x}^{2} H_{y}^{\frac{1}{2}}\left(\mathbb{R}^2\right)}\left[\log \left(1+\frac{\|u\|_{L_{x}^{2} H_y^{s}(\mathbb{R}^2)}}{\|u\|_{L_{x}^{2} H_y^{\frac{1}{2}}(\mathbb{R}^2)}}\right)\right]^{\frac{1}{2}}.
\end{align*}
Set $N(\tau): = \|u(\tau)\|_{L_x^2 H_y^s (\mathbb{R}^2)}^{q'}$. We plug (\ref{2.3})  into the above estimate, since $\|u(t)\|_{L_x^2 H_y^{\frac{1}{2}}(\mathbb{R}^2)}$ is uniformly bounded in $t$, we have
\begin{align*}
N(t) & \leq F(t) : =  C + C \int_{0}^{t} N(t') \left[\log \left(2+N(t')\right)\right]^{\frac{q'(p-1)}{2}} dt'.
\end{align*}
This nonlinear Gronwall inequality can be solved classically in the new unknown $F$ by 
\begin{align*}
F^{\prime}(t)=C N(t) [\log (2+N(t))]^{\frac{q'(p-1)}{2}} \leq C(2+F(t)) [\log (2+F(t))]^{\frac{q'(p-1)}{2}},
\end{align*}
so that
\begin{align*}
\frac{d}{d t} [\log (2+F(t))]^{1-\frac{q'}{2}(p-1)} \leq C.
\end{align*}
Integrating from 0 to $\tau$, we conclude
\begin{align*}
[\log (2+N(\tau))]^{1-\frac{q'}{2}(p-1)} \leq [\log (2+F(\tau))]^{1-\frac{q'}{2}(p-1)} \leq  [\log \left(2+C\right)]^{1-\frac{q'}{2}(p-1)}+C \tau,
\end{align*}
which implies
\begin{align*}
N(\tau) \leq C \mathrm{e}^{C \tau^{\frac{5-p}{7-3p}}}.
\end{align*}
This shows that $N(\tau) $ remains bounded if $\tau$ remains bounded. From the blow up criterion (\ref{3.31}), we can deduce the global existence of the local solution. Furthermore, arguing as in the proof of the local problem, we can prove that two solutions in $C(\mathbb{R},\mathcal{H}^s)$ to (\ref{1.1}) which coincide at $t=0$ must coincide on $\mathbb{R}$, and this  yields the uniqueness of the global solution to (\ref{1.1}) in $\mathcal{H}^{s}$($\frac{1}{2}< s \leq 1$). Similarly, we can show that the flow map $u_0 \in \mathcal{H}^s \mapsto u \in C\left([-T, T], \mathcal{H}^s\right)$ is continuous for every $T>0$.  The proof of the global well-posedness of (\ref{1.1}) in $\mathcal{H}^{s}$($\frac{1}{2}< s \leq 1$) is complete.\\\\
\textbf{Case 2: } $s=\frac{1}{2}$.\\
Firstly, we prove the global existence of weak solutions in $ \mathcal{H}^{\frac{1}{2}}$. Let $u_0 \in \mathcal{H}^{\frac{1}{2}}$, and we approximate $u_0$ by a sequence $\left(u_{0}^{n}\right)$ in $\mathcal{H}^s$ with $\frac{1}{2}<s \leq 1$. We denote by $\left(u_{n}(t)\right)$ the solutions in $C(\mathbb{R}; \mathcal{H}^s)$ corresponding to the initial data $\left(u_{0}^{n}\right)$. By the energy conservation, the $\mathcal{H}^{\frac{1}{2}}$ norm of $u_n(t)$ remains bounded for any $t \in \mathbb{R}$, and thus $\partial_{t}u_n(t)$ remains bounded in $H_{x,y}^{-1}(\mathbb{R}^2)$, then we can deduce that there exists a subsequence of $u_n(t)$ converging weakly to $u(t)$ in $\mathcal{H}^{\frac{1}{2}}$, locally uniformly in $t$. For the convenience of notations, we still denote this subsequence by $\left(u_{n}(t)\right)$. By the Rellich theorem, $u_n(t)$ converges strongly to $u(t)$ in $L_{loc}^{\gamma}(\mathbb{R}^2)$ for every $2\leq \gamma < 6$, and it is easy to verify that $u$ is a weak solution of (\ref{1.1}).\\\\
Then we prove the uniqueness of the weak solution by using an approach first introduced by V. I. Yudovich \cite{ref13} in the study of the 2D Euler equation. Let $u_1(t)$ and $u_2(t)$ be two weak solutions to (\ref{1.1}) with $u_1(0) = u_2(0) = u_0$. Let $k>2$ and $\tau>0$, and we recall the following notations
\begin{align*}
\frac{1}{q'} = 1-\frac{1}{q} = \frac{5-p}{4}, \quad \frac{1}{r'}= 1-\frac{1}{r} = \frac{p}{2}.
\end{align*}
Then by Lemma \ref{lemma 2.1}, we have 
\begin{align*}
 & \quad \|u_1(\tau)-u_2(\tau)\|_{L_{x,y}^2} \\ & \leq C \left\||u_1|^{p-1}u_1 - |u_2|^{p-1}u_2 \right\|_{L_t^{q'}\left((0,\tau);L_x^{r'} L_y^2\right)} \\ & \leq C \left\|(|u_1|^{p-1}+|u_2|^{p-1})|u_1-u_2|\right\|_{L_t^{q'}\left((0,\tau);L_x^{r'} L_y^2\right)} \\ &  \leq C \left\|(|u_1|^{p-1+\frac{1}{k}}+|u_2|^{p-1+\frac{1}{k}})|u_1-u_2|^{1-\frac{1}{k}}\right\|_{L_t^{q'}\left((0,\tau);L_x^{r'} L_y^2\right)} \\ & \leq C \left\|(\|u_1\|_{L_y^{2k(p-1)+2}}^{p-1+\frac{1}{k}} + \|u_2\|_{L_y^{2k(p-1)+2}}^{p-1+\frac{1}{k}}) \|u_1-u_2\|_{L_y^2}^{1-\frac{1}{k}}\right\|_{L_t^{q'} \left((0,\tau); L_x^{r'}\right)} \\ & \leq C \left\|(\|u_1\|_{L_x^2 L_y^{2k(p-1)+2}}^{p-1+\frac{1}{k}} + \|u_2\|_{L_x^2 L_y^{2k(p-1)+2}}^{p-1+\frac{1}{k}}) \|u_1-u_2\|_{L_{x,y}^2}^{1-\frac{1}{k}}\right\|_{L_t^{q'}(0,\tau)} \\ &  \leq C \left(\|u_1\|_{ L_t^{\infty}((0,\tau); L_x^2 L_y^{2k(p-1)+2})}^{p-1+\frac{1}{k}}+\|u_2\|_{ L_t^{\infty}((0,\tau); L_x^2 L_y^{2k(p-1)+2})}^{p-1+\frac{1}{k}} \right)\left\| \|u_1-u_2\|_{L_{x,y}^2}^{1-\frac{1}{k}}\right\|_{L_t^{q'}(0,\tau)}\\ & \leq C k^{\frac{p-1}{2}} \, \tau^{\frac{1}{kq'}} \left(\|u_1\|_{ L_t^{\infty}((0,\tau); L_x^2 H_y^{\frac{1}{2}})}^{p-1+\frac{1}{k}}+\|u_2\|_{ L_t^{\infty}((0,\tau); L_x^2 H_y^{\frac{1}{2}})}^{p-1+\frac{1}{k}} \right) \left\|u_1-u_2\right\|_{L_t^{q'}\left((0,\tau); L_x^2 L_y^2\right)}^{1-\frac{1}{k}} \\ & \leq C k^{\frac{p-1}{2}} \tau^{\frac{1}{kq'}} \left\|u_1-u_2\right\|_{L_t^{q'}\left((0,\tau); L_x^2 L_y^2\right)}^{1-\frac{1}{k}}.
\end{align*}
The penultimate inequality above comes from Lemma \ref{lemma 2.4}.\\\\
Set 
\begin{align*}
g(\tau) : = \int_{0}^{\tau} \|u_1-u_2\|_{L_{x,y}^2}^{q'} dt,
\end{align*}
then the above inequalities imply
\begin{align*}
g^{\prime}(\tau) \leq C k^{\frac{2(p-1)}{5-p}} \tau^{\frac{5-p}{4k}} (g(\tau))^{1-\frac{1}{k}},
\end{align*}
thus we have
\begin{align*}
g(\tau) \leq (Ck^{\frac{3p-7}{5-p}} \tau^{\frac{5-p}{4k}+1})^{k}.
\end{align*}
The right hand side of the latter inequality goes to zero as $k$ goes to infinity for any $\tau>0$. This provides the uniqueness of the weak solution.\\\\
It remains to prove that the weak solution $u$ is strongly continuous in time with values in $\mathcal{H}^{\frac{1}{2}}$, and that it depends continuously on the initial data $u_0$.  First, by the mass conservation and the weak convergence of $u_n(t)$ in $L_{x,y}^2(\mathbb{R}^2)$, we can de deduce that
\begin{align*}
\|u(t)\|_{L_{x,y}^2(\mathbb{R}^2)} \leq \lim_{n\to\infty}  \|u_n(t)\|_{L_{x,y}^2(\mathbb{R}^2)} = \lim_{n\to\infty} \|u_0^n\|_{L_{x,y}^2(\mathbb{R}^2)} = \|u_0\|_{L_{x,y}^2(\mathbb{R}^2)}
\end{align*}
for any $t \in \mathbb{R}$. By reversing time and using uniqueness, we obtain the converse inequality for any $t \in \mathbb{R}$, thus we have
\begin{equation}
\label{3.4}
\|u(t)\|_{L_{x,y}^2(\mathbb{R}^2)} = \lim_{n\to\infty}  \|u_n(t)\|_{L_{x,y}^2(\mathbb{R}^2)} = \lim_{n\to\infty} \|u_0^n\|_{L_{x,y}^2(\mathbb{R}^2)} = \|u_0\|_{L_{x,y}^2(\mathbb{R}^2)}.
\end{equation}
From (\ref{3.4}) and the fact that $\|u_n(t)\|_{L_{x,y}^2(\mathbb{R}^2)} = \|u_0^n\|_{L_{x,y}^2(\mathbb{R}^2)}$, we can deduce the uniform convergence of $\|u_n(t)\|_{L_{x,y}^2(\mathbb{R}^2)}$
to $\|u(t)\|_{L_{x,y}^2(\mathbb{R}^2)}$. Since the weak convergence of $u_n(t)$ to $u(t)$ in $L_{x,y}^2(\mathbb{R}^2)$ is locally uniform in $t$, we infer that $u_n(t)$ converges strongly to $u(t)$ in $L_{x,y}^2(\mathbb{R}^2)$, locally uniformly in $t$.\\\\
Then we prove the strong continuity in time of $u$ in $\mathcal{H}^{\frac{1}{2}}$. Since $u_n(t)$ and $u(t)$ are uniformly bounded in $\mathcal{H}^{\frac{1}{2}}$, by the interpolation and the locally uniform convergence of $u_n(t)$ to $u(t)$ in $L_{x,y}^2(\mathbb{R}^2)$, we can deduce that $u_n(t)$ converges strongly to $u(t)$ in $L_{x,y}^{p+1}(\mathbb{R}^2)$, locally uniformly in $t$. Thus, by the weak convergence of $u_n(t)$ to $u(t)$ in $\mathcal{H}^{\frac{1}{2}}$, the strong convergence of $u_n(t)$ to $u(t)$ in $L_{x,y}^{p+1}(\mathbb{R}^2)$ and the energy conservation,  we have
\begin{align*}
\frac{1}{2}\|u(t)\|_{\mathcal{H}^{\frac{1}{2}}}^2 \pm \frac{1}{p+1}  \|u(t)\|_{L_{x,y}^{p+1}(\mathbb{R}^2)}^{p+1} & \leq \frac{1}{2} \liminf_{n\to\infty}\|u_n(t)\|_{\mathcal{H}^{\frac{1}{2}}}^2 \pm \frac{1}{p+1} \lim_{n\to\infty} \|u_n(t)\|_{L_{x,y}^{p+1}(\mathbb{R}^2)}^{p+1} \\ & \leq \lim_{n\to\infty}\left( \frac{1}{2}\|u_n(t)\|_{\mathcal{H}^{\frac{1}{2}}}^2 \pm \frac{1}{p+1}  \|u_n(t)\|_{L_{x,y}^{p+1}(\mathbb{R}^2)}^{p+1}\right) \\ & = \lim_{n\to\infty}\left( \frac{1}{2}\|u_0^n\|_{\mathcal{H}^{\frac{1}{2}}}^2 \pm \frac{1}{p+1}  \|u_0^n\|_{L_{x,y}^{p+1}(\mathbb{R}^2)}^{p+1}\right) \\ & =  \frac{1}{2}\|u_0\|_{\mathcal{H}^{\frac{1}{2}}}^2 \pm \frac{1}{p+1}  \|u_0\|_{L_{x,y}^{p+1}(\mathbb{R}^2)}^{p+1}
\end{align*}
for any $t \in \mathbb{R}$. By reversing time and using uniqueness, we obtain the converse inequality for any $t \in \mathbb{R}$, and then we have the following convergence
\begin{equation}
\label{3.5}
\frac{1}{2}\|u(t)\|_{\mathcal{H}^{\frac{1}{2}}}^2 \pm \frac{1}{p+1}  \|u(t)\|_{L_{x,y}^{p+1}}^{p+1}   = \lim_{n\to\infty}\left( \frac{1}{2}\|u_n(t)\|_{\mathcal{H}^{\frac{1}{2}}}^2 \pm \frac{1}{p+1}  \|u_n(t)\|_{L_{x,y}^{p+1}}^{p+1}\right).
\end{equation}
which is locally uniform in $t$. Combining (\ref{3.5}) and the locally uniform convergence of $u_n(t)$ to $u(t)$ in $L_{x,y}^{p+1}(\mathbb{R}^2)$, we can deduce that $u_n(t)$ converges strongly to $u(t)$ in $\mathcal{H}^{\frac{1}{2}}$ and the convergence is locally uniform in $t$.  This implies the strong continuity of $u(t)$ in $\mathcal{H}^{\frac{1}{2}}$. The continuity of the flow map can be proved similarly. The proof of the global well-posedness of (\ref{1.1}) in $\mathcal{H}^{\frac{1}{2}}$ is complete.
\end{proof}
Moreover, we infer the global well-posedness of the following quadratic half wave Schrödinger equations
\begin{equation}
\label{3.51}
\begin{aligned}
i \partial_{t}u + (\partial_x^2-|D_y|) u & =\pm |u| u, \quad(x, y) \in \mathbb{R} \times \mathbb{R},\\
u(0,x,y) & = u_0(x,y)
\end{aligned}
\end{equation}
in the higher order Sobolev space $H^2(\mathbb{R}^2)$, which can be deduced as a corollary of Theorem \ref{theorem 3.1}.
\begin{corollary}
\label{corollary 3.2}
Given $u_0 \in H^2(\mathbb{R}^2)$, there exists a unique global solution $u \in C(\mathbb{R};H^2(\mathbb{R}^2))$ to (\ref{3.51}) with $u(0) = u_0$. Also, for every $T>0$, the flow map $u_0 \in H^2(\mathbb{R}^2) \mapsto u \in C\left([-T, T], H^2(\mathbb{R}^2)\right)$ is continuous. 
\end{corollary}
\begin{proof}
Firstly, we shall prove the existence of the local solution in $H^2(\mathbb{R}^2)$ with $u(0) = u_0$. We set
\begin{align*}
 Y_{T} & :=C\left([-T, T]; H^2(\mathbb{R}^2)\right) , \\ \|u\|_{Y_{T}} & :=\max _{t \in[-T, T]}\|u(t)\|_{H^2(\mathbb{R}^2)}.
\end{align*}
We recall the definition
\begin{align*}
\Phi_{2}(u)(t):=\mathrm{e}^{i t (\partial_x^2-|D_y|)} u_{0} \pm i \int_{0}^{t} \mathrm{e}^{i\left(t-t^{\prime}\right) (\partial_x^2-|D_y|)}\left(\left|u\left(t^{\prime}\right)\right| u\left(t^{\prime}\right)\right) d t^{\prime}.
\end{align*}
We estimate directly the $Y_T$ norm of $\Phi_{2}(u)(t)$, combined with the Minkowski inequality and Lemma \ref{lemma a2}, we have
\begin{align*}
\|\Phi_{2}(u)\|_{Y_{T}} \leq \left\|u_{0}\right\|_{H^2(\mathbb{R}^2)}+ C T\|u\|_{Y_{T}}^{2}.
\end{align*}
Similarly, we have 
\begin{align*}
\|\Phi_{2}(u)-\Phi_{2}(\tilde{u})\|_{Y_{T}} \leq K T \max \left(\|u\|_{Y_{T}},\|\tilde{u}\|_{Y_{T}}\right)\|u-\tilde{u}\|_{Y_{T}}.
\end{align*}
Consequently, if $R>0$ is such that
\begin{align*}
R>2 \left\|u_{0}\right\|_{H^2(\mathbb{R}^2)}, \max (C, K) T R \leq \frac{1}{2},
\end{align*}
then $\Phi_2$ is a contraction map on the closed ball $B_R$ of radius $R$ and centered at 0 in $Y_T$, and therefore has a unique fixed point in $B_R$. Thus there exist $T>0$ and a local solution $u \in C\left([-T, T] ; H^2(\mathbb{R}^2)\right)$ to (\ref{3.51}) with $u(0) = u_0$.\\\\
Then we study the global existence of this local solution. From the above derivation, we have the following blow up criterion at $T^{*} > 0$
\begin{equation}
\label{3.6}
\|u(\tau)\|_{H^2(\mathbb{R}^2)}  \xrightarrow{ \tau \to T^{*}} \infty.
\end{equation}
From Theorem \ref{theorem 3.1}, we already have the global well-posedness of (\ref{3.51}) in $H^1(\mathbb{R}^2)$, so we have the boundedness with $t$ of $\|u(t)\|_{H^1(\mathbb{R}^2)}$ in any compact subset of $\mathbb{R}$. We may assume that $\tau>0$, then from the Duhamel's formula and Lemma \ref{lemma a2}, we have
\begin{align*}
\|u(\tau)\|_{H^{2}(\mathbb{R}^2)} & \leq\left\|u_{0}\right\|_{H^{2}(\mathbb{R}^2)}+ \int_{0}^{\tau}\left\|\left|u\left(t\right)\right| u\left(t\right)\right\|_{H^{2}(\mathbb{R}^2)} d t\\ & \leq \left\|u_{0}\right\|_{H^{2}(\mathbb{R}^2)}+ C \int_{0}^{\tau}\left\|\hat{u}_{\eta}(t)(\xi)\right\|_{L_{\xi,\eta}^1(\mathbb{R}^2)}\left\|u\left(t\right)\right\|_{H^{2}(\mathbb{R}^2)} d t.
\end{align*}
At this stage we plug the 2D Brezis-Gallouët inequality on $\mathbb{R}^2$ (\ref{2.41}), then we obtain
\begin{align*}
\|u(\tau)\|_{H^{2}(\mathbb{R}^2)} \leq\left\|u_{0}\right\|_{H^{2}(\mathbb{R}^2)}+ C \int_{0}^{\tau}\|u(t)\|_{H^{1}\left(\mathbb{R}^{2}\right)}\left[\log \left(1+\frac{\|u(t)\|_{H^{2}\left(\mathbb{R}^{2}\right)}}{\|u(t)\|_{H^{1}\left(\mathbb{R}^{2}\right)}}\right)\right]^{\frac{1}{2}}\left\|u\left(t\right)\right\|_{H^{2}(\mathbb{R}^2)} d t.
\end{align*}
Since $\|u(t)\|_{H^{1}(\mathbb{R}^2)}$ is bounded in $[0,\tau]$, by the argument used in the proof of Theorem \ref{theorem 3.1}, we infer that 
\begin{align*}
\|u(\tau)\|_{H^{2}(\mathbb{R}^2)} \leq C \mathrm{e}^{C_{\tau}}.
\end{align*}
Then by (\ref{3.6}), we can deduce the global existence of the solution.
The uniqueness of the global solution and the continuity of the flow map comes from the contraction argument. The proof is complete.
\end{proof}
\begin{remark}
In Theorem \ref{theorem 3.1}, $s$ cannot be extended to arbitrarily large numbers since we cannot expect high regularity on the nonlinear term $|u|^{p-1}u$ for $1<p\leq 2$. Similarly, $H^2(\mathbb{R}^2)$ cannot be improved to $H^3(\mathbb{R}^2)$ in Corollary \ref{corollary 3.2} by the same reason. See also Remark \ref{A.1} and Remark \ref{A.2} for details. Moreover, different from many other well-posedness results, our global well-posedness result implies the uniqueness of the solution in the space where the initial data is located, which is the result of unconditional uniqueness. 
\end{remark}
\begin{remark}
\label{remark 3.4}
The contraction argument used in the proof of Theorem \ref{theorem 3.1} allows us to prove that the flow map $u_{0} \mapsto u(t)$ is Lipschitz continuous on bounded subsets of $\mathcal{H}^s$($\frac{1}{2}<s\leq 1$).
\end{remark}
Then we focus on the Cauchy problem of (\ref{1.01}). In fact, Remark \ref{remark 2.3}, Remark \ref{remark 2.6} and Remark \ref{remark 2.8} allow us to adapt the approach in the proof of Theorem \ref{theorem 3.1} to deduce the global well-posedness of (\ref{1.01}) in $\mathcal{K}^s$ with $\frac{1}{2}\leq s \leq 1$. 
\begin{theorem}
\label{theorem 3.5}
Let $\frac{1}{2}\leq s \leq 1$. Given $u_0 \in \mathcal{K}^s : = L_x^2 H_y^s(\mathbb{R}\times\mathbb{T}) \cap H_x^1 L_y^2(\mathbb{R}\times\mathbb{T})$, then there exists a unique global solution $u \in C\left(\mathbb{R} ; \mathcal{K}^s\right)$ to $(\ref{1.01})$ with $u(0) = u_0$. Moreover, for every $T>0$, the flow map $u_0 \in \mathcal{K}^s \mapsto u \in C\left([-T, T], \mathcal{K}^s\right)$ is continuous.
\end{theorem}
\begin{remark}
As we claim in Remark \ref{remark 3.4}, we also infer the Lipschitz continuity of the flow map $u_0 \mapsto u(t)$ on bounded subsets of $\mathcal{K}^s$($ \frac{1}{2}<s\leq 1 $) in Theorem \ref{theorem 3.5}.
\end{remark}
\section{Final comments}
We show briefly here some other results related to the well-posedness of nonlinear half wave Schrödinger equations on the plane.\\\\
We recall the nonlinear half wave Schrödinger equations on the plane
\begin{equation}
\label{5.1}
\begin{aligned}
i \partial_{t}u + (\partial_x^2-|D_y|) u & = \pm |u|^{p-1} u, \quad(x, y) \in \mathbb{R} \times \mathbb{R},\\
u(0,x,y) & = u_0(x,y).
\end{aligned}
\end{equation}
The equation (\ref{5.1}) is scale-invariant. If $u(t,x,y)$ is a solution to (\ref{5.1}), then
\begin{align*}
u_{\lambda}(t, x, y)=\lambda^{\frac{2}{p-1}} u\left(\lambda^{2} t, \lambda x, \lambda^{2} y\right)
\end{align*}
also satisfies (\ref{5.1}) for all $\lambda > 0$. We observe that
\begin{align*}
\left\|\left.u_{\lambda}\right|_{t=0}\right\|_{\dot{H}_{x}^{2 s_c} L_y^2 }=\left\|u_{0}\right\|_{\dot{H}_{x}^{2 s_c} L_y^2} ,\quad \left\|\left.u_{\lambda}\right|_{t=0}\right\|_{L_{x}^{2} \dot{H}_y^{s_c}}=\left\|u_{0}\right\|_{L_{x}^{2}\dot{H}_y^{s_c}} \text{ for all } \lambda >0
\end{align*}
with
\begin{align*}
s_{c}:=\frac{3}{4}-\frac{1}{p-1}.
\end{align*}
According to the results of nonlinear Schrödinger equations, we may expect some global well-posedness results of (\ref{5.1}) in the energy subcritical case or in the energy critical case. The case $1<p<\frac{7}{3}$ is called the mass subcritical case, and we notice that we have the $\mathcal{H}^{\frac{1}{2}}$ boundedness of solutions in $\mathcal{H}^{\frac{1}{2}}$ with respect to $t$ in the focusing case, so we might expect the global well-posedness results in both the focusing case and the defocusing case.  The case $1<p<5$ is called the energy subcritical case, but for $\frac{7}{3}\leq p<5$, we lose the $\mathcal{H}^{\frac{1}{2}}$ boundedness of solutions in $\mathcal{H}^{\frac{1}{2}}$ with respect to $t$ in the focusing case, so we might only expect to obtain the global well-posedness in the defocusing case for $\frac{7}{3}\leq p \leq 5$.  Moreover, the case $p = 5$ is called the energy critical case.  Since we already have the global well-posedness results in the case of $1<p\leq 2$, we discuss here the well-posedness results and the corresponding problems in the case of $2<p\leq 5$ as follows.\\\\
In fact, for $2<p\leq 5$, by using Lemma \ref{lemma 2.1}, we still have the local well-posedness of (\ref{5.1}) in $\mathcal{H}^{s}$  for some $s>\frac{1}{2}$, but we should add the space $L_{t}^4 \left((-T,T);L_x^\infty H_y^s (\mathbb{R}^2)\right)$ in the study of the Cauchy problem,  this means we do not have the unconditional uniqueness in this case. One can also see the proof of Theorem 1.6 in \cite{ref1} as a reference. However, there is no good way to prove the global well-posedness in any $\mathcal{H}^s$ yet. In fact, in the study of the global well-posedness in $\mathcal{H}^s$ with $s>\frac{1}{2}$, the Brezis–Gallouët argument is not sufficient to control the $L_{x,y}^{\infty}$ norm of solutions. Also, the argument in the proof of uniqueness of the weak solution to (\ref{1.1}) in $\mathcal{H}^{\frac{1}{2}}$ cannot be adapted directly in this case since we do not have $u \in L_t^4 L_x^{\infty} H_y^{\frac{1}{2}}$ for the weak solution $u$ in the energy space $\mathcal{H}^{\frac{1}{2}}$.
\begin{appendix}
\section{Appendix}
In the appendix, we introduce two fundamental lemmas which will be adapted in the proof of Theorem \ref{theorem 3.1} and Corollary \ref{corollary 3.2}.\\\\
The following lemma provides the $H^s(\mathbb{R})$($ 0<s\leq 1 $) estimate of the nonlinear term $|f|^{p-1}f$ with $p>1$.
\begin{lemma}
\label{lemma a1}
Let $p>1$ and $0 <s\leq 1$, then for $f \in H^s(\mathbb{R})$, we have
\begin{equation}
\label{A.1}
\left\||f|^{p-1}f\right\|_{H^s(\mathbb{R})} \leq C \left\|f\right\|_{L^{\infty}(\mathbb{R})}^{p-1} \left\|f\right\|_{H^s(\mathbb{R})}.
\end{equation}
\end{lemma}
\begin{proof}
We decompose the proof into two parts, the first part is the proof of the case of $f \in H^s(\mathbb{R})$ with $0<s< 1$, the second part is the proof of the case of $f \in H^1(\mathbb{R})$.\\\\
1. For $0<s< 1$:\\
We introduce the following equivalent norm of $H^s(\mathbb{R})$($ 0 <s< 1 $)
\begin{align*}
\|f\|_{H^s(\mathbb{R})} : = \left(\int_{\mathbb{R}} \int_{\mathbb{R}} \frac{|f(x+h)-f(x)|^{2}}{|h|^{2 s+1}} d x d h\right)^{\frac{1}{2}}.
\end{align*}
We have
\begin{align*}
\left\||f|^{p-1}f\right\|_{H^s(\mathbb{R})}^2 & = \int_{\mathbb{R}} \int_{\mathbb{R}} \frac{\left||f(x+h)|^{p-1}f(x+h)-|f(x)|^{p-1}f(x)\right|^{2}}{|h|^{2 s+1}} d x d h \\ & \leq C \int_{\mathbb{R}} \int_{\mathbb{R}}\left(|f(x+h)|^{p-1} + |f(x)|^{p-1} \right)^2 \frac{|f(x+h)-f(x)|^{2}}{|h|^{2 s+1}} d x d h \\ & \leq C \left\|f\right\|_{L^{\infty}(\mathbb{R})}^{2(p-1)} \int_{\mathbb{R}} \int_{\mathbb{R}} \frac{|f(x+h)-f(x)|^{2}}{|h|^{2 s+1}} d x d h \\ & \leq C \left\|f\right\|_{L^{\infty}(\mathbb{R})}^{2(p-1)} \left\|f\right\|_{H^s(\mathbb{R})}^2,
\end{align*}
which implies (\ref{A.1}) with $0<s<1$.\\\\
2. For $s=1$:\\
By Leibniz rule, we observe that
\begin{align*}
\left(|f|^{p-1} f\right)^{\prime} = |f|^{p-1} f' + (|f|^{p-1})^{\prime} f = |f|^{p-1} f' + \frac{(p-1)}{2} \left(|f|^{p-1} f' + |f|^{p-3} f^2 \bar{f}^{\prime}\right),
\end{align*}
then we can easily deduce (\ref{A.1}) with $s=1$.
\end{proof}
\begin{remark}
\label{remark a2}
In fact, if $p$ is odd, we can also deduce the estimate as in (\ref{A.1}) for any $s>1$ since we have $|f|^2 = f \bar{f}$. But if $p$ is not odd, we cannot expect the estimate as in (\ref{A.1}) for every $s>1$ since the sigularity will appear in the high order derivative of $|f|^{p-1}$. 
\end{remark}
Then we introduce the following lemma which gives the control of the $H^2(\mathbb{R}^2)$ norm of the quadratic nonlinear term.
\begin{lemma}
\label{lemma a2}
For $g(x,y) \in H^2(\mathbb{R}^2)$, we have
\begin{equation}
\label{A.2}
\left\||g| g \right\|_{H^2(\mathbb{R}^2)} \leq C \left\|g\right\|_{L^{\infty}(\mathbb{R}^2)} \|g\|_{H^2(\mathbb{R}^2)}.
\end{equation}
\begin{proof}
In fact, it is enough to show 
\begin{equation}
\label{A.3}
\left\|\partial_x^2(|g| g )\right\|_{H^2(\mathbb{R}^2)} \leq C \left\|g\right\|_{L^{\infty}(\mathbb{R}^2)} \|g\|_{H^2(\mathbb{R}^2)}.
\end{equation}
By Leibniz rule, we have
\begin{align*}
 \partial_{x}^2 \left(|g|g\right)  = |g| \partial_{x}^2 g + 2 \, \partial_{x} g \, \partial_{x} |g| + g\partial_{x}^2 |g|
\end{align*}
and
\begin{align*}
\partial_{x}^2 |g| = & \, \frac{1}{2} \partial_x \left( |g|^{-1} \bar{g}\partial_x g + |g|^{-1} g \partial_x \bar{g} \right) \\  = &\, -\frac{1}{2}|g|^{-2} \partial_x |g| \left(  \bar{g}  \partial_x g + g   \partial_x \bar{g}\right) + |g|^{-1} |\partial_x g|^2 \\  & +  \frac{1}{2} |g|^{-1} \left(\bar{g} \partial_x^2 g +  g \partial_x^2 \bar{g}\right).
\end{align*}
Since we have $\left|\partial_x |g|\right| \leq \left|\partial_x g\right|$, in order to show (\ref{A.3}), we only have to prove 
\begin{align*}
\|\partial_x g \|_{L^4(\mathbb{R}^2)}^2 \leq  C \|g\|_{L^{\infty}(\mathbb{R}^2)} \|\partial_x^2 g\|_{L^{2}(\mathbb{R}^2)}.
\end{align*}
The above inequality comes from the Gagliardo–Nirenberg interpolation inequality. The proof is complete.
\end{proof}
\begin{remark}
As explained in Remark \ref{remark a2}, (\ref{A.2}) cannot hold for $H^s(\mathbb{R}^2)$ with $s\geq 3$ since the singularity appears in the third derivative of $|g|$.
\end{remark}
\end{lemma}
\end{appendix}

\end{document}